\documentclass[11pt]{amsart}
\usepackage{latexsym,amsfonts,amssymb,graphicx,amsmath}
\usepackage{epstopdf}
\usepackage{amssymb}
\usepackage{enumerate}
\usepackage{enumitem}
\usepackage{booktabs}
\usepackage{stmaryrd}
\usepackage{sidecap}
\usepackage{color}
\usepackage[font=small,labelfont=bf]{caption}
\numberwithin{equation}{section}

\newcommand{\revjj}[1]{\textcolor{black}{{#1}}}
\newcommand{\revj}[1]{\textcolor{black}{{#1}}}
\newcommand{\revm}[1]{\textcolor{black}{{#1}}}

\newcommand{\dO}{\partial\Omega}

\newcommand{\dive}{\operatorname{div}}
\newcommand{\jj}{\mathbf{j}_h}

\newcommand{\T}{\mathcal{T}}

\newtheorem{theorem}{Theorem}
\newtheorem{lemma}{Lemma}
\newtheorem{proposition}{Proposition}

\theoremstyle{definition}

\newtheorem{assumption}{Assumption}
\newtheorem{corollary}{Corollary}
\newtheorem{remark}{Remark}

\begin{document}
%  -------   TITLE   ------    %
\title[Inf-Sup stability of unfitted Stokes  elements]{Inf-sup stability of geometrically unfitted Stokes finite elements}

%  -------   AUTHORS   -------   %

\author[J. Guzm\'an]{Johnny Guzm\'an\textsuperscript{\textdagger}}
\address{\textsuperscript{\textdagger} Division of Applied Mathematics, Brown University, Providence, RI 02912, USA}
\email{johnny\_guzman@brown.edu}
\thanks{Partially supported by NSF through the Division of Mathematical Sciences grant  1318108.}

\author[M.A. Olshanskii]{Maxim Olshanskii\textsuperscript{\textdaggerdbl}}
\address{\textsuperscript{\textdaggerdbl}Department of Mathematics, University of Houston, Houston, TX 77204, USA}
\email{molshan@math.uh.edu}
\email{}
\thanks{Partially supported by NSF through the Division of Mathematical Sciences grant  1522252.}

\maketitle

\begin{abstract}  The paper shows an inf-sup stability property for several well-known 2D and 3D Stokes elements on  triangulations which are not fitted to a given smooth or polygonal domain. The property implies stability
and optimal error estimates for a class of unfitted finite element methods for the Stokes and Stokes interface
problems, such as Nitsche-XFEM or cutFEM. The error analysis is presented for the Stokes problem. All assumptions made in the paper are satisfied once the background mesh is shape-regular and fine enough.
\end{abstract}
\medskip

%   ------- KEYWORDS   ------  %
\smallskip

%   ------   SUBJECT CLASS   ------   %
\date{}

% ---------------------------------------------- %
\section{Introduction}\label{section_intro}
% ---------------------------------------------- %
Unfitted finite element (FE) methods incorporate geometrical information about the domain where the problem is posed
without fitting the mesh to lower dimensional structures such as physical boundaries or internal interfaces.
This is opposite to fitted desretizations such as (isoparametric) traditional FE and isogeometric analysis. The advantage of the unfitted approach is a relative ease of handling propagating interface and geometries defined implicitly, i.e. when a surface
parametrization is not readily available. Prominent classes of unfitted FE are given by XFEM \cite{fries2010extended} and cutFEM \cite{cutFEM} also known as Nitsche-XFEM methods or trace FE in the case of embedded surfaces. In cutFEM, one considers background mesh and FE spaces not tailored to the problem geometry, while numerical integration in FE bilinear forms is performed  over the physical domains $\Omega$ and/or $\dO$ which cut through the background mesh in an arbitrary way.
Effectively, this leads to traces of the ambient FE spaces on the physical domain, where the original problem is posed, and integration over arbitrary cut simplexes.

The idea of  unfitted FE can be followed back at least to the works of Barrett and Elliott \cite{barrett1984finite,barrett1986finite, barrett1987fitted}, where a cut FE method was studied for the planar elliptic problems and elliptic interface problems.
Over the last decades, unfitted FE methods  emerge in a powerful discretization  approach  that has been applying  to the wide range of problems, including problems with interfaces, fluid equations, PDEs posed on surfaces, surface-bulk coupled problems, equations posed on evolving domains, etc, see, e.g., \cite{becker2009nitsche,boffi2003finite,burman2012fictitious,deckelnick2014unfitted,dolbow2009efficient,traceFEM, li1998immersed,li2003new,olshanskii2014eulerian,reusken2014analysis,schott2014new}. %In particular, the unfitted finite elements were recently applied to solving equations posed on evolving domains within purely Eulerian framework, see, e.g., \cite{olshanskii2014eulerian,hansbo2016cut}.
Among important enabling techniques used in unfitted FEM are the Nitsche method for enforcing essential boundary and interface conditions \cite{hansbo2002unfitted}, ghost penalty stabilization \cite{burman2010ghost}, and the properties of trace FE spaces on embedded surfaces~\cite{ORG09}.  We note that many of these developments are accomplished with rigorous stability and convergence analysis of the unfitted FE, which demonstrate both utility and reliability of the approach.

One important application of unfitted FE methods is the numerical simulation of fluid problems with evolving interfaces as
occurs in fluid-structure interaction problems and two-phase flows. If the fluid is treated as incompressible, then the prototypical model suitable for numerical analysis is the stationary (interface) Stokes problem. This paper addresses the question of numerical stability of a certain class of geometrically unfitted Stokes finite elements. Unfitted FE methods for the Stokes problem received recently a closer attention in the literature. In \cite{burman2014fictitious}
optimal order convergence results were shown for the unfitted inf-sup stable velocity--pressure 2D FE with Nitsche treatment of the boundary conditions and ghost-penalty stabilization for triangles cut by $\dO$. This analysis was extended to the Stokes interface problem and $P_1isoP_2-P_1$ elements in \cite{hansbo2014cut}. Optimal order convergence in the energy norm for $P_1^{\rm bubble}-P_1$ unfitted FE using slightly different pressure stabilization
over cut triangles was shown for the Stokes interface problem in \cite{cattaneo2015stabilized}. In  \cite{Reusken2016} the $P_2-P_1$ elements were analysed for the  Stokes interface problem, when the pressure element is enriched to allow for the jump over unfitted interface, while the velocity element is globally continuous.
Globally stabilized unfitted Stokes finite elements, $P_1-P_1$ and $P_1-P_{1}^{\rm disc}$, were studied in \cite{cattaneo2015stabilized,massing2014overlapping,wang2015new}. Other related work
on geometrically unfitted FE for the Stokes problem can be found in \cite{amdouni2012numerical,gross2007extended,legrain2008stability,sauerland2011extended}.

The analysis of inf-sup stable unfitted Stokes elements, however, is not a straightforward extension of the standard results for saddle point problems. In particular, it essentially relies on a certain uniform stability property of the finite element velocity--pressure pair. This property can be found as an assumption (explicitly or implicitly made) in  \cite{burman2014fictitious,cattaneo2015stabilized,hansbo2014cut}. Loosely speaking the following condition on  FE velocity--pressure spaces is required: Assume a family of shape-regular triangulations $\{\T_h\}_{h>0}$ of $\mathbb{R}^2$, and let $\Omega\in \mathbb{R}^2$ be a bounded domain with smooth boundary. Consider
the family of domains $\Omega_h$, where each $\Omega_h$ consists of all triangles from $\T_h$ which are \textit{strictly inside} $\Omega$. Then one requires that the LBB  constants (optimal constants from the FE velocity--pressure inf-sup stability condition) for the domains $\Omega_h$ are uniformly in $h$ bounded away from zero. In the same way the property is formulated in 3D. In section~\ref{section1} we discuss what sort of difficulties one encounters trying to employ common techniques to verify this property.

Recently, in \cite{Reusken2016} the required uniform stability condition was proved for $P_2-P_1$, the lowest order Taylor-Hood element. In this paper, we show the uniform inf-sup stability result for a wider class of elements,
including $P_{k+1}-P_k$, $k\ge 1$, and   $P_{k+d}-P_{k}^{\rm disc}$ for $k\ge0$, $\Omega\subset\mathbb{R}^d$, $d=2,3$, and several other elements, see Section~\ref{s_spaces}.
Following~\cite{Reusken2016} we employ the argument  from \cite{verfurth1984error}. This helps us to formulate  more local condition on FE spaces which are sufficient for the uniform inf-sup stability, but easier to check. Further we show that
this condition is satisfied by a number of popular LBB-stable FE pairs.

The paper also applies the acquired uniform stability result to show the optimal order error estimates of
the unfitted FE method for the Stokes problem. The analysis improves over the available in the literature by
eliminating certain assumptions on how the surface $\dO$ (or an interface in the two-phase fluid case) intersects
the background mesh. Instead, we impose certain assumptions, which are always satisfied once the background mesh is shape-regular and the mesh size is not too coarse with respect to the problem geometry, see section~\ref{s_infsup}  for the assumptions and further discussion in Remark~\ref{R_d1}.

The remainder of the paper is organized as follows. In section~\ref{section1} we define the problem of interest and formulate the
central question we address in this paper about uniform inf-sup stability. Section~\ref{section2} collects necessary preliminaries and auxiliary results. Here we present the unfitted finite element method for the Stokes problem. Further we formulate assumptions sufficient for the main uniform stability result. Section~\ref{s_error} shows how the well-posedness and optimal order error
estimates for the unfitted FE method follow from our assumptions.  In section~\ref{s_spaces} we give the examples of velocity and pressure spaces satisfying the assumptions.

% ---------------------------------------------- %
\section{Problem setting}\label{section1}
% ---------------------------------------------- %
Consider the Stokes problem posed on a bounded  domain with Lipschitz boundary $\Omega\subset \mathbb{R}^d$, $d=2,3$,
\begin{equation}\label{Stokes}
\left\{
\begin{aligned}
    -\Delta u+ \nabla p &= f \qquad &\mbox{in } &\Omega, \\
                                           \dive u &=0            & \mbox{in } &\Omega, \\
                            u&=0            & \mbox{on }& \partial \Omega.
\end{aligned}
\right.
\end{equation}
Vector function $u\in \left[H^1_0(\Omega)\right]^d$ and $p\in L^2(\Omega)/\mathbb{R}$ are the weak solution to \eqref{Stokes}, having the physical meaning of fluid velocity and normalized kinematic pressure.

Assume there is a domain $S\supset\Omega$, and we let $\{\mathcal{T}_h\}_{h>0}$ be an admissible family of triangulations of $S$. We are interested in a finite element method for \eqref{Stokes} using spaces of piecewise polynomial functions with
respect to $\T_h$. Note that we make no assumption on how $\Omega$ overlaps with $\T_h$, i.e. $\dO$ may cut through tetrahedra or triangles from $\T_h$ in an arbitrary way.

In the next section we give details of the finite element method. Now we formulate the stability condition, which is crucial for the analysis of this method (and likely many other unfitted FE methods for \eqref{Stokes}).
Consider the set of all strictly internal simplexes and define the corresponding subdomain of $\Omega$:
\[
\mathcal{T}_h^{i} := \{T\in \mathcal{T}_h: T \subset \Omega \},\qquad \Omega_h^{i}  := \mbox{Int} \Big(\bigcup_{T\in\mathcal{T}_h^{i}} \overline{T}\Big).
\]
For background finite element velocity and pressure spaces $V_h$ and $Q_h$, consider their restrictions on $\Omega_h^i$, that is $V_h^i=V_h \cap \left[H_0^1(\Omega_h^i)\right]^d$ and $Q_h^i=Q_h\cap L^2_0(\Omega_h^i)$, $L^2_0(\Omega_h^i):=\{q\in L^2(\Omega_h^i)~:~\int_{\Omega_h^i} q\,dx=0\}$, and define
\[
\theta_h:=\inf_{q \in Q_h^i}\sup_{v \in V_h^i} \frac{\revm{\int_{\Omega_h^i} q\dive v \,dx}}{\|v\|_{H^1(\Omega_h^i)}\|q\|_{L^2(\Omega_h^i)}}.
\]
\textit{We are interested in the following condition}:
\begin{equation}\label{LBBi}
0<\inf\limits_{h<h_0} \theta_h,
\end{equation}
for some positive $h_0$.

Note that standard arguments based on the Ne\v{c}as inequality and Fortin's projection operator, cf. \cite{brezzi2008mixed}, cannot be applied in a straightforward way to yield  \eqref{LBBi}  for inf-sup stable elements  (e.g., for Taylor-Hood element).
For the reference purpose recall the Ne\v{c}as inequality:
\begin{equation}\label{Necas}
C_N(\Omega_h^i)\|q\|_{L^2(\Omega_h^i)}\le \sup_{v \in\left[H_0^1(\Omega_h^i)\right]^d} \frac{\revm{\int_{\Omega_h^i} q\dive v \,dx}}{\|v\|_{H^1(\Omega_h^i)}}\quad\forall~q\in L^2_0(\Omega_h^i).
\end{equation}
Since $\Omega_h^i$ is Lipschitz for any given $\T_h$, the inequality holds with some domain dependent constant $C_N(\Omega_h^i)>0$, see, e.g., \cite{bogovskii1979solution,galdi2011introduction}.
However, we are not aware of a result in the literature which implies that $C_N(\Omega_h)$ are uniformly  bounded from below by a positive constant independent of $h$. For example, the well-known argument for proving \eqref{Necas} is based on the decomposition of a Lipschitz domain into a finite number of strictly star shaped domains (see Lemma~II.1.3 in \cite{galdi2011introduction}) and applying the result of Bogovskii~\cite{bogovskii1979solution} in each of the star domains.
However, the number of the star domains in the decomposition of $\Omega_h^i$ may infinitely grow for $h\to0$
even if $\dO$ is smooth and $\{\T_h\}_{h>0}$ is shape-regular, which would drive the lower bound for $C_N(\Omega_h^i)$ to zero. Alternatively, the recent analysis from \cite{bernardi2015continuity} provides a lower bound for
$C_N(\Omega_h^i)$ if there exist diffeomorphisms $\Phi_h\,:\,\Omega_h^i\to \Omega$  with uniformly bounded
$W^{1,\infty}(\Omega_h^i)$ norms. We do not see how to construct such diffeomorphisms (note that $\dO_h^i$
is not necessarily a graph of a function in the natural coordinates of $\dO$). Additional difficulty stems from the observation that $\T_h^i$ does not necessarily inherit a macro-element structure that $\T_h$ may possess.  This said, we shall look for a different approach to verify \eqref{LBBi}.

We end this section noting that the finite element method and the analysis of the paper can be easily extended to the Stokes \textit{interface} problem, a prototypical model of two-phase incompressible fluid flow. However, we are not adding these extra details to the present report.

% ---------------------------------------------- %
\section{Finite element method}\label{section2}
% ---------------------------------------------- %

\subsection{Preliminaries}
 We adopt the convention that elements $T$ and element edges (also faces in 3D) are open sets. We use over-line symbol to refer to their closure. For each simplex $T\in \mathcal{T}_h$, let $h_T$ denote its diameter and define the global parameter of the triangulation by $h = \max_{T} h_T$. We assume that $\mathcal{T}_h$ is shape regular, i.e. there exists $\kappa>0$ such that for every $T\in \mathcal{T}_h$  the radius $\rho_T$ of its inscribed sphere satisfies
\begin{equation}\label{shaperegularity}
\rho_T>h_T/\kappa.
\end{equation}

The set of elements cutting the interface $\Gamma \equiv \partial \Omega$, and restricted to $\Omega$  are also of interest. They are defined by:
\begin{align*}
%\mathcal{T}_h^{i}    &:= \{T\in \mathcal{T}_h: T \subset \Omega \},\\
\mathcal{T}_h^{\Gamma} &:= \{T\in \mathcal{T}_h: \revm{\mbox{meas}_2(T \cap \Gamma)>0}  \},\\
\mathcal{T}_h^{e}&:=\{ T \in  \mathcal{T}_h: T \in  \mathcal{T}_h^{i} \text{ or } T \in  \mathcal{T}_h^{\Gamma} \}.
\end{align*}

In particular for $T \in \mathcal{T}_h^{\Gamma}$ we denote  $T_{\Gamma} =\overline{T} \cap \Gamma$. Observe that the definition of $\mathcal{T}_h^{\Gamma}$ guarantees that $\sum_{T\in\mathcal{T}_h^{\Gamma}}|T_\Gamma| = |\Gamma|$. Under these definitions we define the \revm{$h$-dependent} domains
\[
%\Omega_h^{i}  := \mbox{Int} \Big(\bigcup_{T\in\mathcal{T}_h^{i}} \overline{T}\Big),\quad
\Omega_h^{\Gamma}  := \mbox{Int} \Big(\bigcup_{T\in\mathcal{T}_h^{\Gamma}} \overline{T}\Big), \quad
\Omega_h^{e}  := \mbox{Int} \Big(\bigcup_{T\in\mathcal{T}_h^{e}} \overline{T}\Big).
\]
Note that \revm{$\overline{\Omega_h^e}=\overline{\Omega_h^i} \cup \overline{\Omega_h^\Gamma}$} and that $\Omega_h^i \subset \Omega \subset \Omega_h^e$.
For these domains define sets of faces:
\begin{align*}
\mathcal{F}_h^{i} & := \{ F: F \text{ is an interior face of } \mathcal{T}_h^i \}, \\
\mathcal{F}_h^{\Gamma}  & := \{ F: F \text{ is a face of } \mathcal{T}_h^\Gamma, F \not\subset \partial \Omega_h^e  \} ,    \\
\mathcal{F}_h^{e} & := \{ F: F \text{ is an interior face of } \mathcal{T}_h^e \}.
\end{align*}

Now we can define finite element spaces.
A space of continuous functions on $\Omega_h^{e}$ which are polynomials of degree $k$ on each $T\in \mathcal{T}_h^{e}$ is denoted by $W^k_h$.
The spaces of discontinuous and continuous pressure spaces are given by
\begin{align*}
Q_h^{\rm disc}&=\{ q \in L^2(\Omega_h^e): q|_T \in P^{k_p}(T), \forall~ T \in \mathcal{T}_h^e\}, \\
Q_h^{\rm cont}&=Q_h^{\rm disc} \cap H^1(\Omega_h^e) .
\end{align*}
Throughout this paper we will consider either $Q_h=Q_h^{\rm disc}$ (for $k_p \ge 0$) or $Q_h=Q_h^{\rm cont}$ (for $k_p \ge 1$).  We will denote the finite element velocity space by $V_h \subset [H^1(\Omega_h^e)]^d$, and  we will assume %\footnote{\bf I introduced $k_p$ and $k_u$, since MINI and P2isoP1-P1 elements didn't fit the framework}
\[ \big(W_h^{k_u}\big)^d \subset V_h \subset\big( W_h^{s}\big)^d
\]
for some integer $s\ge k_u\ge1$.
In section~\ref{s_infsup}, we introduce a more technical assumption~\ref{assumption2} that our pair of spaces $\{Q_h,V_h\}$ has to satisfy. Then, later we give examples of pairs that satisfy all necessary assumptions.  For example, if $Q_h=Q_h^{\rm disc}$ then $V_h$ can be the space of continuous piecewise polynomials of degree $k_p+d$; and if $Q_h=Q_h^{\rm cont}$ then $V_h$ can be the space of continuous piecewise polynomials of degree $k_p+1$. We give more examples of spaces satisfying our assumptions in section~\ref{s_spaces}.

\subsection{Finite element method} We will use the notation $(v,w)=\int_{\Omega} v w\, dx$. Introduce the mesh-dependent bilinear forms
\[
a_h(u_h, v_h):=(\nabla u_h, \nabla v_h)+s_h(u_h, v_h)+ \jj(u_h, v_h) +\eta j_h(u_h, v_h),
\]
with
\begin{align*}
s_h(u,v)&=-\int_{\Gamma} \{(n\cdot\nabla u) \cdot v + (n\cdot\nabla v) \cdot u\} ds,\\
j_h(u,v)&=  \sum_{T \in \mathcal{T}_h^\Gamma } \frac{1}{h_T} \int_{T_{\Gamma}} u \cdot v ds,\\
\jj(u,v)&=\sum_{F \in \mathcal{F}_h^\Gamma} \sum_{\ell = 1}^s h_F^{2 \ell-1} \int_{F}  \left[\partial_{n}^\ell u\right ] \left[\partial_{n}^\ell v\right] ds;
\end{align*}
and
\[
b_h(p_h, v_h):= -(p_h, \dive v_h)+r_h(p_h, v_h),
\]
with
\[
r_h(p,v)=\int_{\Gamma} p \, v \cdot n\, ds.
\]
Here and further $\partial_n^\ell q$  on face $F$ denotes the derivative of order $\ell$ of $q$ in direction $n$, where $n$ is normal to $F$; and
$[\phi]$ denotes the jump of a quantity $\phi$ over a face $F$.

We can now define the numerical method:
Find $u_h \in V_h$ and $p_h \in Q_h$ such that
\begin{equation}\label{FEM}
\left\{
\begin{aligned}
a_h(u_h, v_h)+b_h(p_h, v_h)&= (f, v_h), \\
b_h(q_h, u_h) - J_h (p_h, q_h) &=0,
\end{aligned}
\right.
\end{equation}
for all $v_h\in V_h,~q_h\in Q_h,$ where
\[
J_h(q,p)=\sum_{F \in \mathcal{F}_h^\Gamma} \sum_{\ell = 0}^{k_p} h_F^{1+2 \ell} \int_{F}  \left[\partial_{n}^\ell q\right ] \left[\partial_{n}^\ell p\right] ds.
\]
 The unfitted FE method in \eqref{FEM} was introduced in \cite{burman2014fictitious}.

 Pressure solutions to both \eqref{Stokes} and \eqref{FEM} are defined up to an additive constant. It is convenient to assume that the restriction of $p_h$ on $\Omega_h^i$ is from  $L^2_0(\Omega_h^i)$. We shall fix one particular $p$ solving \eqref{Stokes} later.

Before  proceeding with the analysis, we briefly discuss the role of different terms in the finite element formulation \eqref{FEM}.
First note that all \revm{volume} integrals in \eqref{FEM} are computed over physical domains $\Omega$ and $\Gamma$ rather than computational domain $\Omega_h^e$.
The gradient and div-terms appear due to the integration by parts in a standard weak formulation of the Stokes problem. Since finite element
velocity trial and test functions do not satisfy homogenous Dirichlet conditions strongly on $\Gamma$,  the integration by parts brings the $s_h$ and $r_h$ terms to the formulation. The $-\int_{\Gamma} (n\cdot\nabla v) \cdot u ds$ integral in $s_h$ is added to make formulation symmetric. It vanishes for $u$, the Stokes equations solution. The same is true for the $r_h$ term in the continuity equation in \eqref{FEM}.  The penalty term $j_h(u_h, v_h)$ weakly enforces the Dirichlet boundary conditions for $u_h$, as common for the Nitsche method, with a parameter $\eta=O(1)$. The terms $\jj(u_h, v_h)$ and $J_h (p_h, q_h)$ are added for the numerical stability of the method: we need  $\jj(u_h, v_h)$ to gain control over normal velocity derivatives in $s_h$, and we need $J_h$ for pressure stability over cut triangles. In practice, both  $\jj$ and $J_h$ can be scaled by additional stabilization parameters of $O(1)$ order; we omit this detail here.

We note that the unfitted FEM analyzed in the paper is closely related to the extended finite element method (XFEM). Indeed, the trace space of background finite element functions on the domain $\Omega$ can be alternatively described as a FE space spanned over nodal shape functions from $\Omega$ and further enriched by certain  degrees of freedom tailored to $\dO$. Hence the results of this paper can be as well considered as the analysis of a certain class of XFEM methods for the Stokes problem.

Next section proves the key result for getting numerical stability and optimal order error estimates for the unfitted finite element method \eqref{FEM}.
%Note that if $Q_h=Q_h^{\rm cont}$ then
%\begin{equation*}
%\sum_{F \in \mathcal{F}_h^\Gamma} \sum_{\ell =0}^n h_F^{1+2 \ell} \int_{F}  \left[\partial_{n}^\ell q\right ] \left[\partial_{n}^\ell p\right] ds=\sum_{F \in \mathcal{F}_h^\Gamma} \sum_{\ell = 1}^n h_F^{1+2 \ell} \int_{F}  \left[\partial_{n}^\ell q\right ] \left[\partial_{n}^\ell p\right] ds.
%\end{equation*}

\section{Stability} \label{s_infsup}
We need to define some norms and semi-norms.   First we define the mesh-dependent norm for the velocity
\begin{equation*}
\|u\|_{V_h}^2=|u|_{H^1(\Omega)}^2+ j_h(u,u)+\jj(u,u).
\end{equation*}
Note that due to the boundary term $j_h$, the functional $\|u\|_{V_h}$  defines a norm on $V_h$ equivalent to the $H^1(\Omega)$ norm, $\|u\|_{H^1(\Omega)}\lesssim\|u\|_{V_h}\lesssim \revm{h^{-1}_{\rm min}}\|u\|_{H^1(\Omega)}$, \revm{$h_{\rm min}=\min_{T\in\T_h^e}h_T$}.
We need a set of all tetrahedra intersected by $\Gamma$ together with all tetrahedra from $\Omega$ touching those:
\begin{equation*}
\widetilde{\mathcal{T}}_h^{\Gamma}=\{ T: T \in \mathcal{T}_h^{\Gamma}  \text{ or } T \subset \Omega,  \overline{T} \cap \overline{\Omega_h^\Gamma } \neq \emptyset \},
\end{equation*}
and also
\begin{equation*}
\widetilde{\Omega}_h^{\Gamma}  := \mbox{Int} \Big(\bigcup_{T\in\widetilde{\mathcal{T}}_h^{\Gamma}} \overline{T}\Big). \\
\end{equation*}

For a generic set of tetrahedra $\T\subset\T_h$ denote $\omega(\T)\subset\T_h$ the set of all tetrahedra having at least one vertex in $\T$.
We need the following assumptions on how well the geometry is resolved by the mesh.

\begin{assumption}\label{assumption1}  For any $T\in \mathcal{T}_h^\Gamma$ we assume that the set $W(T)=\mathcal{T}_h^i \cap \omega \left(\omega(T)\right)$ is not empty.
\end{assumption}

We note that the assumption can be weaken by allowing in $W(T)$ neighbors of $T$  of degree $L$, with some finite and mesh independent $L\ge2$.

Given $T \in \mathcal{T}_h^\Gamma$ we associate an arbitrary but fixed $K_T \in W(T)$, \revm{which can be reached from $T$ by crossing  faces in $\mathcal{F}_h^\Gamma$.  More precisely,  there exists simplices $T=K_1, K_2, \ldots, K_M = K_T$ with $K_j \in \mathcal{T}_h^\Gamma$ for $j<M$.} The number $M$ is uniformly bounded and only depends on the shape regularity of the mesh. Note that by \eqref{shaperegularity} there exists a constant $c$ only depending on the shape regularity constant $\kappa$ such that $ \frac{1}{c}  h_T \le h_{K_T} \le c h_T$. \revjj{For $T \in \mathcal{T}_h^i$ we define $K_T= T$.}

\revjj{
\begin{assumption}\label{assumption1b}
Let $F \in \mathcal{F}_h^\Gamma$ with $F= \partial T_1 \cap \partial T_2$. We assume $K_{T_2}$ can be reached from $K_{T_1}$ by crossing a finite, independent of $h$, number of faces of tetrahedra from $\mathcal{T}_h^{i}$.
\end{assumption}
We recall that we assume that $\Omega$  is Lipschitz. %We emphasize this property ina separate assumption.
}

\begin{remark}\rm\label{R_d1}
One can check that the assumptions~\ref{assumption1}--\ref{assumption1b} are  satisfied if $h$ is sufficiently small and the minimal angle condition \eqref{shaperegularity} holds. This is an improvement of the available analysis of unfitted finite elements which commonly imposes a further restriction on how interface intersects $\T_h$. In 2D this extra assumption is formulated as \revjj{follows}: $\dO$ does not
intersect any edge from $\mathcal{F}^e_h$ more than one time, see, e.g. \cite{hansbo2002unfitted}.
An analogous restriction was commonly assumed in 3D. One easily builds an example showing that this extra assumption is not necessarily true for arbitrary fine
mesh and smooth $\dO$, while enforcing it by `eliminating' \revj{ineligible} elements introduces $O(h^2)$ geometrical
error diminishing possible benefits of using higher order elements.  We do not need this extra assumption.

The assumptions \ref{assumption1}--\ref{assumption1b} also allow local mesh refinement.
\end{remark}

We will make use of the following well known scaled trace inequality.
\begin{equation}\label{oldtrace}
\|v\|_{L^2(\partial T)}\le C(h_T^{-\frac12}\|v\|_{L^2(T)}+h_T^{\frac12}\|\nabla v\|_{L^2(T)}),\quad ~~\forall~v\in H^1(T).\end{equation}

We will also need a local trace inequality for parts of $\Gamma$. We give the proof of the result only assuming that the boundary is Lipschitz in the appendix.  Under various stronger assumptions the following result  was proved in \cite{hansbo2002unfitted,hansbo2004finite,chernyshenko2013non,reusken2014analysis}.

\begin{lemma}\label{lemmatrace}
Under assumption that $\Omega$  is Lipschitz we have the following inequality for every  $T\in \mathcal{T}_h^{\Gamma}$
\begin{equation}\label{trace}
\|v\|_{L^2(T\cap\Gamma)}\le C(h_T^{-\frac12}\|v\|_{L^2(T)}+h_T^{\frac12}\|\nabla v\|_{L^2(T)}),\quad ~~\forall~v\in H^1(T),
\end{equation}
with a constant $C$ independent of $v$, $T$,  how $\Gamma$ intersects $T$, and $h<h_0$ for some arbitrary but fixed $h_0$.
\end{lemma}

One can show the following stability result.
\begin{lemma}\label{l_coerc}
For $\eta$ sufficiently large and $h\le h_0$ for sufficiently small $h_0$, there exists a mesh-independent constant $c_0>0$ such that
\begin{equation}\label{stabilityu}
c_0 \|v_h\|_{V_h}^2 \le  a_h(v_h, v_h)\qquad\forall~v_h\in V_h.
\end{equation}
\end{lemma}
\begin{proof}
To show \eqref{stabilityu} we need the following estimate, see Lemma 5.1 in \cite{massing2014stabilized}:
For any $T_1,T_2$ from $\mathcal{T}_h^{\Gamma}$ sharing a face $F=\overline{T_1}\cap\overline{T_2}$ it holds
\begin{equation}\label{aux6a}
\|q\|^2_{L^2(T_1)}\le C\left(\|q\|^2_{L^2(T_2)}+ \sum_{\ell = 0}^m h_F^{1+2 \ell} \int_{F}  \left[\partial_{n}^\ell q\right ]^2ds\right),\quad\forall~q\in P_m(T_1)\times P_m(T_2),
\end{equation}
with a constant $C$ depending only on the shape regularity of $\mathcal{T}_h$ and polynomial degree $m$.
Thanks to FE inverse inequality, \eqref{aux6a} and Poincare inequality, we have  for any $T_1,T_2$ from $\widetilde{\mathcal{T}}_h^{\Gamma}$ sharing a face $F=\overline{T_1}\cap\overline{T_2}$ the following estimate
\begin{equation}\label{aux6}
\begin{split}
\|\nabla v_h\|^2_{L^2(T_1)}&\le Ch_{F}^{-2}\|v_h-\alpha \|^2_{L^2(T_1)} \le
 C\left(h_{T_2}^{-2}\|v_h-\alpha \|^2_{L^2(T_2)}+ \sum_{\ell = 1}^s h_F^{-1+2 \ell} \int_{F}  \left[\partial_{n}^\ell v_h\right ]^2\right)\\& \le
 C\left(\|\nabla v_h\|^2_{L^2(T_2)}+ \sum_{\ell = 1}^s h_F^{-1+2 \ell} \int_{F}  \left[\partial_{n}^\ell v_h\right ]^2\right),\qquad\forall~v_h\in V_h,
 \end{split}
\end{equation}
where we take $\alpha =|T_2|^{-1}\int_{T_2}v_h\,ds$.  This inequality is also found in Proposition 5.1  in \cite{massing2014stabilized} in the case $V_h=[W_h^1]^{\revjj{d}}$.
%\textbf{\small Maxim says: Actually, \eqref{aux6} constitutes Proposition 5.1. in \cite{massing2014stabilized}, but I am not convinced with the proof given there - please check.}
Thanks to assumption~\ref{assumption1} the estimate
\eqref{aux6} implies
\begin{equation}\label{aux7}
\|\nabla v_h\|^2_{L^2(\Omega_h^e)}\le C(\|\nabla v_h\|^2_{L^2(\Omega)}+ \jj(v_h,v_h)\,).
\end{equation}

%One makes use of the trace inequality:
%\begin{equation}\label{trace}
%\|v\|_{L^2(T\cap\Gamma)}\le C(h_T^{-\frac12}\|v\|_{L^2(T)}+h_T^{\frac12}\|\nabla v\|_{L^2(T)})\quad \forall~T\in \mathcal{T}_h^{\Gamma},~~\forall~v\in H^1(T).
%\end{equation}
%with a constant $C$ independent of $v$, $T$ and how $\Gamma$ intersects $T$.

%Under assumption ~\ref{assumption1a}, the result in \eqref{trace} is proved in  \cite{chernyshenko2013non}.
%Without any special notice we shall also use the particular case of \eqref{trace}, where the norm on the left-hand side
%is replaced by $\|v\|_{L^2(\partial T)}$.
Further, one uses the Cauchy-Schwarz inequality, trace inequality \eqref{trace} and the FE inverse inequality to estimate
\[
\begin{split}
|s_h(v_h,v_h)|&=\left|\int_{\Gamma} (n\cdot\nabla v_h) \cdot v_h ds\right| \le \sum_{T\in\mathcal{T}_h^{\Gamma}}
\|\nabla v_h\|_{L^2(T_\Gamma)} \|v_h\|_{L^2(T_\Gamma)}\\ &\le \frac1{2\eta}\sum_{T\in\mathcal{T}_h^{\Gamma}}
h_T \|\nabla v_h\|_{L^2(T_\Gamma)}^2+\frac\eta2 j_h(v_h,v_h)\\
&\le \frac{C}{2\eta}\sum_{T\in\mathcal{T}_h^{\Gamma}}
(\|\nabla v_h\|_{L^2(T)}^2+h_T^2\|\nabla^2v_h\|_{L^2(T)}^2) +\frac\eta2 j_h(v_h,v_h)
\\
&\le \frac{C}{2\eta}\|\nabla v_h\|_{L^2(\Omega_h^e)}^2 +\frac\eta2 j_h(v_h,v_h).
\end{split}
\]
Combining this with \eqref{aux7} and choosing $\eta$ sufficiently large, but independent of $h$, proves the lemma.
\end{proof}

We need to define the scaled semi-norms for the pressure:
\begin{align*}
|p|_{H_{h,i}^1}^2&=\sum_{T \in \mathcal{T}_h^i} h_T^2 \|\nabla p\|_{L^2(T)}^2+ \sum_{F \in  \mathcal{F}_h^{i}} h_F \|[p]\|_{L^2(F)}^2,\\
|p|_{H_{h,e}^1}^2&=\sum_{T \in \mathcal{T}_h^e} h_T^2 \|\nabla p\|_{L^2(T)}^2+ \sum_{F \in  \mathcal{F}_h^{e}} h_F \|[p]\|_{L^2(F)}^2.
\end{align*}

\begin{assumption}\label{assumption2}
Assume that there exists a constant $\beta>0$ independent of $h$ and only depending on polynomial degree of finite element spaces and the shape regularity of $\mathcal{T}_h$ such that
\begin{equation}\label{assumption2inq}
\beta |q|_{H_{h,i}^1} \le \sup_{v \in V_h^i} \frac{\revm{\int_{\Omega_h^i} q\dive v \,dx}}{\|v\|_{H^1(\Omega_h^i)}} \quad \forall q \in Q_h,
\end{equation}
where $V_h^i=V_h \cap \left[H_0^1(\Omega_h^i)\right]^d$.
\end{assumption}

We  also need the following extension result.  A proof of this result is given in the appendix.
\begin{lemma}\label{lemmaextension}
For every $q \in Q_h$ there exists a $E_h q \in Q_h^{\rm disc}$ such
\begin{equation*}
E_h q=q \text{ on } \Omega_h^i
\end{equation*}
and
\begin{equation}\label{extension1}
|E_h q|_{H_{h,e}^1} \le C |q|_{H_{h,i}^1}.
\end{equation}
\end{lemma}

Using the degrees of freedom of piecewise linear functions one can show the following  result.%\footnote{Referee asks us for a proof of this. Indeed, in \cite{Reusken2016} the results is somehow "hidden" and uses further assumptions on $\Gamma$.}
 %found in \cite{Reusken2016}.%\footnote{\bf I checked the journal version of \cite{Reusken2016}, and it looks to me that
%this and other results in \cite{Reusken2016} we are citing are proved in 3D.}
 %PLEASE CHECK: Note that that the result in \cite{Reusken2016} is given for dimension $d=2$ but the result for $d=3$ follows a similar argument.
\begin{lemma}
For every $v_h \in W_h^1$ there exists a unique decomposition
\begin{equation}\label{decomp}
v_h=\pi_1 v_h+ \pi_2 v_h,
\end{equation}
where $\pi_i v_h \in W_h^1$ for $i=1,2$,  $\pi_2 v_h$ is supported in $\widetilde{\Omega}_h^{\Gamma}$ and such that
\begin{equation} \label{decomposition2}
\pi_2 v_h=v_h  \quad \text{ on } \Omega_h^{\Gamma}
\end{equation}
and
\begin{equation}\label{decomposition1}
\sum_{T \in \widetilde{\mathcal{T}}_h^\Gamma} \frac{1}{h_T^2} \| \pi_2 v_h\|_{L^2(T)}^2  \le C  \sum_{T \in \mathcal{T}_h^\Gamma} \frac{1}{h_T^2} \|v_h\|_{L^2(T)}^2.
\end{equation}
The constant $C$ is independent of $v_h$ and $h$ and only depends on the shape regularity of the mesh.  In particular, note that this implies  $\pi_1v_h \in V_h^i$.
\end{lemma}

\begin{proof}
For a set of tetrahedra $\tau$, $V(\tau)$
denotes the set of all vertices of tetrahedra from $\tau$.
For $v_h \in W_h^1$, one defines $\pi_2 v_h(x)=v_h(x)$ for all $x\in V(\mathcal{T}_h^\Gamma)$ and  $\pi_2 v_h(x)=0$ for all
 $x\in V(\T_h\setminus\mathcal{T}_h^\Gamma)$. It is clear that $\pi_2 v_h=v_h$ on $\Omega_h^{\Gamma}$ and $\pi_1v_h \in V_h^i$. For any $ T\in\widetilde{\mathcal{T}}_h^\Gamma$ let $\widetilde\omega(T)= \omega(T)\cap\mathcal{T}_h^\Gamma$.
Thanks to the shape regularity assumption, we have for any $\widetilde T \in \widetilde{\mathcal{T}}_h^\Gamma$:
\[
\begin{split}
  h_{\widetilde T}^{-2} \| \pi_2 v_h\|_{L^2(\widetilde T)}^2&\le C\, h_{\widetilde T}^{-2}|\widetilde T| \sum_{x\in V(\widetilde T)}|\pi_2 v_h(x)|^2 =
  h_{\widetilde T}^{-2}|\widetilde T| \sum_{x\in V(\widetilde T) \cap V(\mathcal{T}_h^\Gamma)}|\pi_2 v_h(x)|^2\\ &= h_{\widetilde T}^{-2}|\widetilde T| \sum_{x\in V(\widetilde T) \cap V(\mathcal{T}_h^\Gamma)}|v_h(x)|^2
  \le C\, \sum_{T\in\widetilde\omega(\widetilde{T})}h_T^{-2}|T|\sum_{x\in V(T)}|v_h(x)|^2\\ &\le C\, \sum_{T\in\widetilde\omega(\widetilde{T})}h_T^{-2}\|  v_h\|_{L^2(T)}^2.
\end{split}
  \]
Summing over all $ \widetilde{T} \in\widetilde{\mathcal{T}}_h^\Gamma$ and using shape regularity again  we prove the result in
\eqref{decomposition1}.
\end{proof}

The following theorem shows the LBB stability result for the internal domain $\Omega_h^i$ and so proves the key uniform
bound \eqref{LBBi}. Note again that $\Omega_h^i$ is not an $O(h^2)$ approximation of a smooth domain and there is no uniform in $h$ result concerning decomposition of $\Omega_h^i$ into a union of a finite number of star-shaped domains. The latter is a standard assumption for proving the differential counterpart of this finite element condition, see, e.g., \cite{galdi2011introduction}.
This result is  crucial for the stability and convergence analysis of the unfitted FE method \eqref{FEM}. For the lowest order Taylor-Hood element, the proof of the following result is found in \cite{Reusken2016}.
We follow a similar argument, but extend the result so it can be applied to higher order elements in two and three dimensions.

\begin{theorem}\label{infsup}
Suppose Assumptions \ref{assumption1}-- \ref{assumption2} hold.  Then, there exists a constant $\theta>0$  and a constant $h_0$ such that   for all $q \in Q_h$ with $\int_{\Omega_h^i} q dx=0$  we have the following result for $h \le h_0$
\begin{equation}\label{LBB}
\theta \|q\|_{L^2(\Omega_h^i)} \le  \sup_{v \in V_h^i} \frac{\revm{\int_{\Omega_h^i} q\dive v \,dx}}{\|v\|_{H^1(\Omega_h^i)}}.
\end{equation}
The constant $\theta>0$ is independent of  $q$ and $h$.
\end{theorem}

\begin{proof}
Let $\psi=E_h q$ given by Lemma~\ref{lemmaextension} and let   $c_h=\frac{1}{|\Omega|} \int_{\Omega} E_h q$. Using that $\Gamma$ is Lipschitz,
there exists a $v \in [H_0^1(\Omega)]^{\revjj{2}}$ with the following properties, cf. \revm{\cite{bogovskii1979solution,galdi2011introduction}}:
\begin{equation}\label{Bogov}
\dive v = \psi-c_h\quad \text{on}~~ \Omega
\end{equation}
and
\begin{equation}\label{inq1}
\|v\|_{H^1(\Omega)} \le C\, \|\psi-c_h\|_{L^2(\Omega)}.
\end{equation}

Extend $v$ by zero to all of $\Omega_h^e$. Let $v_h\revm{\in W^1_h}$ be the Scott-Zhang interpolant of $v$ \revm{and $v_h|_{\partial \Omega^e_h}=0$}. We will write $(v,w)_e=\int_{\Omega_h^e} v w dx$.
With the help of \eqref{Bogov} and the decomposition \eqref{decomp}, we obtain
\begin{equation}\label{aux*}
\|\psi-c_h\|_{L^2(\Omega)}^2=(\dive v, \psi)= (\dive \pi_1 v_h, \psi)_{e} +(\dive (v- v_h), \psi)_{e}+ (\dive \pi_2 v_h, \psi)_{e}.
\end{equation}

Integration by parts over each $T\in\T_h^e$ gives
\begin{equation*}
(\dive (v-v_h), \psi)_e=  -\sum_{T \in \mathcal{T}_h^e} \int_{T}  (v-v_h)  \cdot \nabla \psi dx- \sum_{F \in \mathcal{F}_h^e} \int_F [\psi] (v-v_h) \cdot n ds.
\end{equation*}
We proceed by applying the Cauchy-Schwarz inequality,  elementwise trace inequality, and the definition of the $H^1_{h,e}$ norm.  This gives the bound
\begin{equation*}
(\dive (v-v_h), \psi)_e \le C\left(\sum_{T \in \mathcal{T}_h^e} (\frac{1}{h_T^2} \|v-v_h\|_{L^2(T)}^2 +\|\nabla (v-v_h)\|_{L^2(T)}^2) \right)^{1/2} |\psi|_{H_{h,e}^1}.
\end{equation*}
Using the approximation properties of the Scott-Zhang interpolant, \eqref{inq1} and \eqref{extension1}, we have
\begin{equation}\label{auxA}
(\dive (v-v_h), \psi)_e  \le C \|\psi-c_h\|_{L^2(\Omega)} |\psi|_{H_{h,e}^1} \le C \|\psi-c_h\|_{L^2(\Omega)} |q|_{H_{h,i}^1}.
\end{equation}

In a similar fashion,  but now using inverse FE estimates instead of  approximation results, and recalling that $\operatorname{supp}(\pi_2 v_h)\subset\widetilde{\Omega}_h^\Gamma$, we show
\begin{equation}\label{auxB2}
(\dive \pi_2 v_h, \psi)_e \le C\left(\sum_{T \in \widetilde{\mathcal{T}}_h^\Gamma} \frac{1}{h_T^2} \| \pi_2 v_h\|_{L^2(T)}^2  \right)^{1/2} |\psi|_{H_{h,e}^1}
\le C\left(\sum_{T \in \mathcal{T}_h^\Gamma} \frac{1}{h_T^2} \|v_h\|_{L^2(T)}^2  \right)^{1/2} |\psi|_{H_{h,e}^1}.
\end{equation}
Note the following Friedrich's type FE inequality:
\[
h_T^{-2}\|v_h\|^2_{L^2(T)}+h_T^{-1}\|v_h\|^2_{L^2(\partial T)}\le C (\|\nabla v_h\|^2_{L^2(T)}+h_T^{-1}\|v_h\|^2_{L^2(F)})\quad\forall~T\in\T_h,~F~\text{is a face of}~T.
\]
We apply the above inequality elementwise and use $v_h=0$ on $\partial\Omega^e_h$ to show that
\begin{equation}\label{inq3}
\left(\sum_{T \in \mathcal{T}_h^\Gamma} \frac{1}{h_T^2} \| v_h\|_{L^2(T)}^2  \right)^{1/2} \le C \|\nabla v_h\|_{L^2(\Omega_h^\Gamma)} \le C \, \|\nabla v\|_{L^2(\Omega_h^e)}=C \, \|\nabla v\|_{L^2(\Omega)}.
\end{equation}
In the last inequality we used the stability of the Scott-Zhang interpolant.
Hence, using \eqref{inq1} we get from  \eqref{auxB2}--\eqref{inq3} the estimate
\begin{equation}\label{auxB}
(\dive \pi_2 v_h, \psi)_e \le  C \|\psi-c_h\|_{L^2(\Omega)} |q|_{H_{h,i}^1}.
\end{equation}
The last term on the right hand side of \eqref{aux*} we handle as follows:
\[
(\dive \pi_1 v_h, \psi)_{e}=(\dive \pi_1 v_h, \psi)_{L^2(\Omega_h^i)}\le  \|\pi_1 v_h\|_{H^1(\Omega_h^i)} \sup_{w \in V_h^i} \frac{(\dive w,  q)_{\revm{L^2(\Omega_h^i)}}}{\|w\|_{H^1(\Omega_h^i)}}.
\]

Now we bound $ \|\pi_1 v_h\|_{H^1(\Omega)} $,
\begin{equation*}
\|\pi_1 v_h\|_{H^1(\Omega)} \le ( \|\pi_2 v_h\|_{H^1(\widetilde{\Omega}_h^\Gamma)} + \| v_h\|_{H^1(\Omega)} ).
\end{equation*}
Using inverse estimates, \eqref{decomposition1} and \eqref{inq3} we get
\begin{equation*}
 \|\pi_2 v_h\|_{H^1(\widetilde{\Omega}_h^\Gamma)} \le C\|v\|_{H^1(\Omega_h^e)}=C\|v\|_{H^1(\Omega)}
\end{equation*}
Hence, the stability of the Scott-Zhang interpolant and \eqref{inq1} imply
\begin{equation}\label{auxC}
\|\pi_1 v_h\|_{H^1(\Omega)} \le C \, \|\psi-c_h\|_{L^2(\Omega)}.
\end{equation}

Therefore, we get from \eqref{auxA}, \eqref{auxB}, \eqref{auxC} and \eqref{aux*} the upper bound
\begin{equation*}
 \|\psi-c_h\|_{L^2(\Omega)}  \le C \left( \sup_{w \in V_h^i} \frac{(\dive w,  q)_{\revm{L^2(\Omega_h^i)}}}{\|w\|_{H^1(\Omega_h^i)}}+|q|_{H_{h,i}^1}\right).
\end{equation*}
Using assumption~\ref{assumption2} we get
\begin{equation}\label{auxD}
 \|\psi-c_h\|_{L^2(\Omega)}  \le C  \sup_{w \in V_h^i} \frac{(\dive w,  q)_{\revm{L^2(\Omega_h^i)}}}{\|w\|_{H^1(\Omega_h^i)}}.
\end{equation}
Finally,  note that
\begin{equation*}
\|c_h\|_{L^2(\Omega)} \le |\Omega|^{1/d} |c_h| = |\Omega|^{\revj{-1+1/d}} \left|\int_{\Omega} \psi dx\right|=|\Omega|^{\revj{-1+1/d}}   \left|\int_{\Omega \setminus \Omega_h^i} \psi dx\right|.
\end{equation*}
The last equality holds since $\psi=q$ in $\Omega_h^i$  and $\int_{\Omega_h^i}q\,dx=0$.
After applying Cauchy-Schwarz inequality and using that $|\Omega \setminus \Omega_h^i|^{\revj{1/d}} \le h^{\revj{1/d}}$ we have that
\begin{equation*}
\|c_h\|_{L^2(\Omega)} \le C h^{1/\revj{d}} \|\psi\|_{L^2(\Omega)}.
\end{equation*}

Hence, using the triangle inequality in \eqref{auxD} and assuming $h$ is sufficiently small we have
\begin{equation*}
 \|\psi\|_{L^2(\Omega)}  \le C  \sup_{w \in V_h^i} \frac{(\dive w,  q)_{\revm{L^2(\Omega_h^i)}}}{\|w\|_{H^1(\Omega_h^i)}}.
\end{equation*}
We note that the constant $C$ is independent of $h$ and $q$.
The result now follows after noting that $\|q\|_{L^2(\Omega_h^i)} \le \|\psi\|_{L^2(\Omega)}$ and letting  $\theta =\frac{1}{C}$.
\end{proof}
\medskip
\begin{corollary}\label{c_b_stab} If assumptions~\ref{assumption1}--\ref{assumption2} hold true, the following stability condition is satisfied
by the $b_h$ and $J_h$ forms of the finite element method \eqref{FEM},
\begin{equation}\label{b_stab}
c_b \|q\|_{L^2(\Omega)} \le  \sup_{v \in V_h} \frac{b_h(v, q)}{\|v\|_{V_h}}+J_h^{\frac12}(q,q)\quad\forall~q \in Q_h,~~s.t.~ q|_{\Omega_h^i}\in L^2_0(\Omega_h^i).
\end{equation}
The constant $c_b>0$ is independent of  $q$ and $h$.
\end{corollary}
\begin{proof} Fix some $q \in Q_h$, such that $q|_{\Omega_h^i}\in L^2_0(\Omega_h^i)$.
Using \eqref{aux6a}, assumption~\ref{assumption1} and the finite overlap argument, one shows
\[
\|q\|_{L^2(\Omega)}^2 \le c\,(\|q\|_{L^2(\Omega_h^i)}^2+J_h(q,q)).
\]
Thanks to the uniform inf-sup property from  Theorem~\ref{infsup} there exists $v\in V_h$ with $\text{supp}(v)\subset\Omega_h^i$ such that
\begin{equation}\label{aux9}
\|q\|_{L^2(\Omega)}^2 \le c\,\Big(\frac{(\dive v,  q)^2}{\|v\|^2_{H^1(\Omega_h^i)}}+J_h(q,q)\Big).
\end{equation}
Using $v=0$ in $\Omega_h^\Gamma$ and applying the FE inverse inequalities we show
\[
\jj(v,v)=\sum_{F \in\mathcal{F}^\Gamma_h,~s.t.~F\subset\partial\Omega_h^i} \sum_{\ell = 1}^s h_F^{2 \ell-1} \int_{F}  \left[\partial_{n}^\ell v\right]^2\,\le
C \sum_{T \in \widetilde{\T}_h^\Gamma\cap\T^i_h}\|\nabla v\|^2_{L^2(T)}\, ds \le C\|v\|^2_{H^1(\Omega_h^i)}.
\]
This estimate  and $j_h(v,v)=0$ for $v\in V_h$ with $\text{supp}(v)\subset\Omega_h^i$ imply the uniform equivalence $\|v\|_{V_h}\simeq \|v\|_{H^1(\Omega_h^i)}$. Using this in \eqref{aux9} yields
 \[
 \|q\|_{L^2(\Omega)}^2 \le c\,\Big(\frac{(\dive v,  q)^2}{\|v\|^2_{V_h}}+J_h(q,q)\Big).
 \]
 Finally, we note  that $(\dive v,  q)=b_h(v, q)$  if $\text{supp}(v)\subset\Omega_h^i$. This completes the proof.
\end{proof}

\section{Well posedness and error estimates} \label{s_error}
One \revjj{easily} verifies that  $a_h$ is continuous
\begin{equation*}%\label{contA}
a_h(u,v)\le C_a\|u\|_{V_h}\|v\|_{V_h}\qquad\forall~u,\,v\in  V_h,
\end{equation*}
with some $C_a>0$ independent of $h$ and the position of $\Gamma$.
The continuity and coercivity of the $a_h(u, v)$ form (Lemma~\ref{l_coerc}) and the inf-sup stability of the $b_h(v, q)$ form (Corollary~\ref{c_b_stab}) readily
imply the stability for the bilinear form of the finite element method~\eqref{FEM} with respect to the product norm,
\begin{equation}\label{infsupA}
C_s\|u_h,p_h\|\le \sup_{\{v,q\}\in V_h\times Q_h}\frac{A_h(u_h,p_h;\,v,q)}{\|v,q\|}\qquad\forall~ \{u_h,p_h\}\in V_h\times Q_h,
\end{equation}
with some $C_s>0$ independent of $h$ and the position of $\Gamma$ and
\[
\begin{split}
A_h(u,p;\,v,q)&:=a_h(u, v)+b_h(v, p)+b_h(u, q)-J_h(p,q),\\
\|v,q\|&:=\left(\|v\|^2_{V_h}+\|q\|_{L^2(\Omega)}^2+J_h(q,q)\right)^{\frac12}.
\end{split}
\]
The proof of \eqref{infsupA} extends standard arguments, cf., e.g., \cite{ern2013theory}, for $J_h\neq0$. For completeness
we sketch the proof here. For given $\{u_h,p_h\}\in  V_h\times Q_h$, thanks to \eqref{b_stab}, one can find $z\in V_h$ such that $\|z\|_{V_h}=\|p_h\|_{L^2(\Omega)}$ and
\[
\begin{split}
c_b \|p_h\|_{L^2(\Omega)}^2 &\le b_h(z, p_h)+J_h^{\frac12}(p_h,p_h)\|p_h\|_{L^2(\Omega)}\\ &=A_h(u_h,p_h;\,z,0)-a_h(u_h, z)+J_h^{\frac12}(p_h,p_h)\|p_h\|_{L^2(\Omega)}\\
&\le A_h(u_h,p_h;\,z,0)+\frac{C_a^2}{c_b}\|u_h\|^2_{V_h}+\frac{c_b}4\|p_h\|_{L^2(\Omega)}^2
+\frac{1}{c_b}J_h(p_h,p_h)+\frac{c_b}4\|p_h\|_{L^2(\Omega)}^2.
\end{split}
\]
Combining this inequality with
\[
J_h(p_h,p_h)+c_0\|u_h\|^2_{V_h}\le A_h(u_h,p_h;\,u_h,{-p_h}),
\]
we get
\[
c\,\|u_h,p_h\|^2\le A_h(u_h,p_h;\,u_h+\alpha z,{-p_h}),
\]
for a suitable $\alpha>0$ and a constant $c\,>0$ depending only on $c_b$, $C_a$, and $c_0$. Inequality \eqref{infsupA} follows by noting $\|u_h,p_h\|\ge \frac{1}{1+\alpha}\|v,q\|$, with $v=u_h-\alpha z$, $q={-p_h}$.
\medskip

One verifies that  $A_h$ is continuous
\begin{equation}\label{contA}
A_h(u,p;\,v,q)\le C_c\|u,p\|\|v,q\|\qquad\forall~\{u,p\},\,\{v,q\}\in  V_h\times Q_h,
\end{equation}
with some $C_c>0$ independent of $h$ and the position of $\Gamma$.
Note also that $A_h$ is symmetric. Therefore, by the  {Banach--Ne\v{c}as--Babu\v{s}ka} theorem {(see, e.g., Theorem~2.6 in \cite{ern2013theory})} the problem~\eqref{FEM} is well-posed and its solution satisfies the stability bound
\[
\|u_h,p_h\|\le C_s^{-1}\|f\|_{V_h'}.
\]

Further in this section we assume  that the solution to the Stokes problem is sufficiently smooth, i.e., $u \in H^{s+1}(\Omega)$ and $p \in H^{k_p+1}(\Omega)$. Since we are assuming that $\Gamma$ is Lipschitz there exist extensions of $u$ and $p$, which we also denote by $u$, $p$, such that $u \in  H^{s+1}(S)$ and $p \in H^{k_p+1}(S)$ (see \cite{SteinBook}). We  let $I_h u$ be the Scott-Zhang interpolant of $u$ onto $\left[W_h^{k_u}\right]^d$. We also let $I_h  p$ be the Scott-Zhang interpolant of $p$ in the case $Q_h=Q_h^{\rm cont} $ and the $L^2$ projection onto discontinuous piecewise polynomials of degree $k_p$ if $Q_h=Q_h^{\rm disc}$. For the pressure interpolant we can always assume $(I_h  p)|_{\Omega_h^i}\in L_0^2(\Omega_h^i)$ by choosing
a suitable additive constant in the definition of $p$.
Applying trace inequalities \eqref{oldtrace} and \eqref{trace}, standard approximation properties of $I_h$,   and extension results one obtains the approximation property in the product norm:
\begin{equation}\label{approxA}
\|u-I_hu,p-I_hp\|\le C\, \left(\,h^{\min\{k_u,k_p+1\}} (\|u\|_{H^{k_u+1}(\Omega)}+ \|p\|_{H^{k_p+1}(\Omega)}) + h^{k_u}\sum_{\ell=k_u+1}^{s+1}h^{\ell-k_u-1} \|u\|_{H^{\ell}(\Omega)}\right).
\end{equation}
We also have the following continuity result and approximation results:
\begin{align}\label{contA2}
A_h(u-I_h u,p-I_h p;\,v,q)&\le C\,\|u-I_h,p-I_h p\|\|v,q\|\\
&\qquad +|s_h(u-I_h u,v)|+|r_h(p-I_h p,v)|,\notag\\
|s_h(u-I_h u,v)|+|r_h(p-I_h p,v)|&\le C\,h^{\min\{k_u,k_p+1\}} (\|u\|_{H^{k_u+1}(\Omega)}+ \|p\|_{H^{k_p+1}(\Omega)})\|v\|_{V_h},\label{contA3}
\end{align}
for all $\{v,q\}\in  V_h\times Q_h$. Here we used \eqref{trace}, \eqref{oldtrace}, \eqref{aux7}.

Denote by $ e_u=u-u_h$ and $e_p=p-p_h$ the finite element error functions.
\revm{Note that for $u \in H^{s+1}(S)$ and $p \in H^{k_p+1}(S)$ the jumps of derivatives in bilinear forms $\mathbf{j}_h$ and $J_h$ vanish. This and the boundary condition $u|_\Gamma=0$ imply
$\mathbf{j}_h(u,v_h)={j}_h(u,v_h)=J_h(p,q_h)=0$ and $s_h(u,v_h)=-\int_\Gamma(n\cdot\nabla u)\cdot v_h$.
Hence, it is easy to see that the method \eqref{FEM} is consistent for $u$ and $p$ sufficiently
smooth as stated above, i.e. \eqref{FEM} is satisfied with $u_h$ replaced by $u$.
Therefore,} the Galerkin orthogonality holds,
\begin{equation} \label{Galerkin}
A_h(e_u,e_p;\,v_h,q_h)=0,
\end{equation}
for all $v_h \in V_h$ and $ q_h \in Q_h$.

The optimal order error estimate in the energy norm is given in the next theorem.

\begin{theorem}\label{Th1} For sufficiently smooth $u,p$ solving \eqref{Stokes} and $u_h,p_h$ solving \eqref{FEM}, the error estimate
holds,
\begin{multline*}
|u-u_h|_{H^1(\Omega)}+ \|p-p_h\|_{L^2(\Omega)} \le \|u-u_h,p-p_h\|\\ \le \,C \left(\,h^{\min\{k_u,k_p+1\}} (\|u\|_{H^{k_u+1}(\Omega)}+ \|p\|_{H^{k_p+1}(\Omega)}) + h^{k_u}\sum_{\ell=k_u+1}^{s+1}h^{\ell-k_u-1} \|u\|_{H^{\ell}(\Omega)}\right),
\end{multline*}
with a constant $C$ independent of $h$ and the position of $\Gamma$ with respect to the triangulation $\T_h$.
\end{theorem}
\begin{proof}
The results follows from the inf-sup stability \eqref{infsupA}, continuity  \eqref{contA2}, Galerkin orthogonality \eqref{Galerkin}, and
approximation properties \eqref{approxA}, \eqref{contA3}, by standard arguments, see, for example, section 2.3 in \cite{ern2013theory}.
\end{proof}

Using the Aubin-Nitsche duality argument one shows the optimal order error estimate for the velocity in $L^2(\Omega)$-norm.
Consider the dual adjoint problem. Let $w\in \left[H^1_0(\Omega)\right]^d$ and $r\in L^2_0(\Omega)$ be the solution to the problem
\begin{equation}\label{Stokes2}
\left\{
\begin{aligned}
    -\Delta w- \nabla r &= e_u \qquad &\mbox{in } &\Omega, \\
                                           \dive w &=0            & \mbox{in } &\Omega, \\
                            w&=0            & \mbox{on }& \partial \Omega.
\end{aligned}
\right.
\end{equation}
We assume that $\Omega$ is such that \eqref{Stokes2} is $H^2$-regular, i.e. for $e_u\in\left[L^2(\Omega)\right]^d$ it holds
$w\in \left[H^2(\Omega)\right]^d$ and $r\in H^1(\Omega)$ and
\[
\|w\|_{H^2(\Omega)}+\|r\|_{H^1(\Omega)}\le C(\Omega)\|e_u\|_{L^2(\Omega)}.
\]
By the standard arguments (section 2.3 in \cite{ern2013theory}) the results in \eqref{infsupA},  \eqref{contA}, \eqref{Galerkin}, \eqref{approxA}, and the above regularity assumption lead to the following theorem.

\begin{theorem}\label{Th2} For sufficiently smooth $u,p$ solving \eqref{Stokes} and $u_h,p_h$ solving \eqref{FEM}, the error estimate
holds,
\begin{equation*}
|u-u_h|_{L^2(\Omega)}\le \,C h \|u-u_h,p-p_h\|,
\end{equation*}
with a constant $C$ independent of $h$ and the position of $\Gamma$ with respect to the triangulation $\T_h$.
\end{theorem}
\medskip

\section{Example of spaces satisfying Assumption \ref{assumption2}}\label{s_spaces}

\subsection{Generalized Taylor-Hood elements}\label{s_TH}
 Consider $Q_h=Q_h^{\rm cont}:=W_h^{k}$ and $V_h=\left[W_h^{k+1}\right]^d$, $k\ge1$. %One has to assume that each $T \in \mathcal{T}_h^i$ has to have at least one node in the interior of $\Omega_h^i$.
In this case, the proof of estimate~\eqref{assumption2inq} from Assumption~\ref{assumption2} is given in section~8 of \cite{brezzi2008mixed} for $d=2$ (two-dimensional case). In three-dimensional case and $k=1$, the result can be found in Lemma~4.23 in \cite{ern2013theory}. Below we extend the proof for all $k\ge1$ in 3D.  We require each $T \in \mathcal{T}_h^i$  to have at least three edges in the interior of $\Omega_h^i$. Note that the proof in \cite{brezzi2008mixed} for $d=2$ does not need a similar assumption. For any edge from the set of internal edges of $\mathcal{T}_h^i$, $E\in \mathcal{E}_h^i$, we denote a unit tangent vector be $t_E$ (any of two, but fixed), $x_E$ is the  midpoint of $E$, and $\omega(E)$ is a set of tetrahedra sharing $E$.
Also we denote by $\phi_E\in W_h^2(\Omega_h^i)$ a piecewise quadratic function such that
$\phi_E(x_E)=1$ and  $\phi_E(x)=0$, where $x$ is any vertex or a midpoint of any other edge from $\mathcal{E}_h^i$. For $p\in Q_h$ we set
\[
v(x)=-\sum_{E\in\mathcal{E}_h^i} h_E^2\phi_E(x)\, [t_E\cdot\nabla p(x)]t_E.
\]
Since the pressure tangential derivative $t_E\cdot\nabla p$ is continuous across faces $F$ that contain $E$, it is easy to see that $v\in V^i_h$.
We compute
\[
\int_{\Omega_h^i} \dive v \, p \,dx= - \int_{\Omega_h^i} v \cdot \nabla p\, dx  =
\sum_{E\in\mathcal{E}_h^i} h_E^2\int_{\omega(E)} \phi_E |t_E\cdot\nabla p|^2\,dx\ge c\, \sum_{E\in\mathcal{E}_h^i} h_E^2\int_{\omega(E)} |t_E\cdot\nabla p|^2\,dx.
\]
The constant $c>0$ in the last inequality depends only on the polynomial degree $k$ and shape regularity condition \eqref{shaperegularity}. From the condition \eqref{shaperegularity} we also infer  $h_E\simeq h_T$ for $T\in\omega(E)$. This gives
after rearranging terms the estimate
\[
\int_{\Omega_h^i} \dive v \, p \,dx\ge c\, \sum_{T\in\mathcal{T}_h^i}\sum_{E\in \overline{T}\cap\Omega_h^i} h_T^2\int_{T} |t_E\cdot\nabla p|^2\,dx\ge c\, \sum_{T\in\mathcal{T}_h^i} h_T^2\int_{T} |\nabla p|^2\,dx=c\|p\|^2_{H^1_{h,i}}.
\]
For the last inequality we used the assumption that at least three edges of the tetrahedra are internal and we apply  the shape regularity condition one more time. Due to the finite element inverse inequalities and the obvious estimate $ |\phi_E|+h_T|\nabla\phi_E|\le c\,$ on $T$, with $E\in \overline{T}$,  we have
\begin{multline*}
\|v\|_{H^1(\Omega_h^i)}^2\le c\int_{\Omega_h^i} |\nabla v|^2\,dx\le
\sum_{T\in\mathcal{T}_h^i}\sum_{E\in \overline{T}\cap\Omega_h^i} h_T^4\int_{T} (|\nabla\phi_E|^2 |\nabla p|^2 + |\phi_E|^2 |\nabla^2 p|^2)\,dx\\
\le
\sum_{T\in\mathcal{T}_h^i} \int_{T} h_T^2(|\nabla p|^2 + h^2_T|\nabla^2 p|^2)\,dx
 \le c \|p\|_{H^1_{h,i}}^2.
\end{multline*}
 This shows \eqref{assumption2inq}.

\subsection{Bercovier-Pironneau element} This is a `cheap' version of the lowest order Taylor-Hood element. In 2D the element in defined in \cite{bercovier1979error}, the 3D version can be found, e.g., in \cite{ern2013theory}.
To define the velocity space, one refines each triangle of $\mathcal{T}_h$ by connecting midpoints on the edges in 2D, while in 3D
one divides a tetrahedron into six tetrahedra by the same procedure.
Then the velocity space consists of piecewise linear continuous function with respect to the refined triangulation, $V_h=\left[W_{h/2}^{1}\right]^d$, and $Q_h=Q_h^{\rm cont}:=W_h^{1}$. For this element, one shows \eqref{assumption2inq} following the lines of the proof of Theorem 8.1 in \cite{brezzi2008mixed} for $k=1$ in 2D or the arguments from the section~\ref{s_TH} with obvious modifications: For example, in the 3D case one   substitutes `edge-bubbles' $\phi_E$  by there P1isoP2 counterparts.

\subsection{$P_{k+2}-P_{k}^{\rm disc}$   (for $d=2$) and $P_{k+3}-P_{k}^{\rm disc}$ (for $d=3$) elements}
We only consider the two dimensional case $d=2$ as the case $d=3$ is similar.
We let $Q_h=Q_h^{\rm disc}$ be the space of piecewise polynomial functions of degree $k$  and let $V_h=\left[W_h^{k+2}\right]^2$.  \revjj{The canonical degrees of freedom of a function $m \in P_{k+2}(T)$ are given by
\begin{alignat*}{1}
\int_T m \, s \, dx & \quad \text{ for all } s \in P_{k-1}(T) \\
\int_E  m \,  q  \, dx & \quad \text{ for all edges }  E \text{ of } T, q \in P_{k}(E) \\
 m(x) & \quad \text{ for all the vertices } x \text{ of } T.
\end{alignat*}
}
To show Assumption  \ref{assumption2} holds in this case, take $q \in Q_h^{\rm disc}$. We can choose $v \in V_h^i$  (\revjj{using the degrees of freedom above}) such that
\begin{equation*}
\int_T v \cdot w\, dx= -h_T^2 \int_{T} \nabla q \cdot w\, dx \quad \text{ for all } w \in P_{k-1}(T) ,
\end{equation*}
and for all $T \in \mathcal{T}_h^i$.
Also for every interior edges $E$ of $\Omega_h^i$
\begin{equation*}
\int_E  r v \cdot n^+ ds= h_E \int_{E}  r (q^+ n^++ q^- n^-) \cdot n^+\, dx   \quad \text{ for all } r \in P_{k}(E) ,
\end{equation*}
where $E=\partial T^+\cap \partial T^-$ and   $T^+, T^- \in \mathcal{T}_h^i$. Also, $n^\pm $ is the outward pointing unit normal of $T^{\pm}$.

To pin down $v \in V_h^i$ we make $v$ vanish on all vertices and have tangential components vanish on all edges. Finally, we make $v \equiv 0$ on $\partial \Omega_h^i$.

\revj{
Using elementwise integration by parts, we get
\begin{equation*}
\int_{\Omega_h^i} \dive v \, q\, dx= \sum_{T \in \mathcal{T}_h^i} - \int_{T} v \cdot \nabla q dx+ \int_{\partial T} q v \cdot n ds.
\end{equation*}
From the construction of $v$, we see that
\begin{equation*}
\int_{\Omega_h^i} \dive v \, q\, dx= \sum_{T \in \mathcal{T}_h^i} h_T^2 \| \nabla q \|_{L^2(T)}^2 + \sum_{E \in \mathcal{E}_h^i} h_F \|[q]\|_{L^2(F)}^2=|q|_{H_h^i}^2
\end{equation*}
}
It is not difficult to show, using a scaling argument that $\|v\|_{H^1(\Omega_h^i)} \le C |q|_{H_h^i}$. From this we see that Assumption \ref{assumption2} holds.

\subsection{Bernardi-Raugel element}

In a similar fashion we can show that the Bernardi-Raugel spaces satisfy Assumption \ref{assumption2}.
The space of Bernardi-Raugel elements consists of  piecewise constant  pressure and for
the velocity one takes  $P^1$ continuous functions enriched with the normal components of the velocity
as a degree of freedom at barycentre face nodes~\cite{bernardi1985analysis}.

\subsection{Mini-Element}
\revm{Let $\Omega\subset\mathbb{R}^2$.} With $k_p=1$,  $Q_h=Q_h^{\rm cont}$ and $V_h= [W_h^1]^2+\{ v: v|_T \in b_T c_T, \text{ where } c_T \in [\mathbf{P}^0(T)]^2, \text{ for all } T \in \mathcal{T}_h^e \}$. Here $b_T$ is the cubic bubble.

To prove~\eqref{assumption2inq} we consider an arbitrary $q \in Q_h$. A simple argument gives
\begin{equation*}
\sum_{ T \in \mathcal{T}_h^i} h_T^2 \|\nabla q \|_{L^2(T)}^2\le C \, \sum_{ T \in \mathcal{T}_h^i} h_T^2 \|\sqrt{b_T} \revm{\nabla q} \|_{L^2(T)}^2.
 \end{equation*}

Integration by parts gives $\|\sqrt{b_T} \nabla q \|_{L^2(T)}^2= - \int_T \dive (b_T \nabla q) q dx$. If we define $w_h \in V_h^i$ in the following way $w_h|_T: =- h_T ^2 b_T \nabla q |_T$ then we have
 \begin{equation*}
 \sum_{ T \in \mathcal{T}_h^i} h_T^2 \|\sqrt{b_T} \nabla q \|_{L^2(T)}^2= \int_{\Omega_h^i} \dive w_h \; q dx.
\end{equation*}
Hence, we get
\begin{equation*}
\sum_{ T \in \mathcal{T}_h^i} h_T^2 \|\nabla q \|_{L^2(T)}^2  \le C \, \sup_{v \in V_h^i} \frac{(\dive v,  q)}{\|v\|_{H^1(\Omega_h^i)}}  \|w_h\|_{H^1(\Omega_h^i)}.
\end{equation*}
Now, using Poincare's inequality
\begin{equation*}
\|w_h\|_{H^1(\Omega_h^i)}^2 \le C \sum_{T \in \mathcal{T}_h^i} \| \nabla w_h\|_{L^2(T)}^2  \le C \sum_{T \in \mathcal{T}_h^i} h_T^4 \|\nabla b_T\|_{L^\infty(T)}^2 \| \nabla q\|_{L^2(T)}^2.
\end{equation*}
Since $h_T^2 \|\nabla b_T\|_{L^\infty(T)}^2 \le C$, we get
\begin{equation*}
\|w_h\|_{H^1(\Omega_h^i)}^2 \le \sum_{ T \in \mathcal{T}_h^i} h_T^2 \|\nabla q \|_{L^2(T)}^2.
\end{equation*}
The result now follows.

\subsection{Generalized conforming Crouzeix-Raviart element}

This element is defined by $Q_h=Q_h^{\rm disc}$ with $k_p=k \ge 1$ for $d=2$ or $k\ge 2$ for $d=3$, we define the velocity space to be $V_h=[W_h^{k+1}]^d+ \{ v:  v|_T \in  b_T  \nabla \mathbf{P}^k(T), \text{ for all } T \in \mathcal{T}_h^e\}$,
where $b_T$ is cubic bubble in two dimensions or quartic bubble in three dimensions.  This $P_{k+1}^{\rm bubble} -P_k^{\rm disc}$ spaces was first introduced in \cite{crouzeix1973conforming}.
The proof of \eqref{assumption2inq} in this case will be similar to that of mini-elment. We leave the details to the reader.

\subsection*{Acknowledgements} We would like to thank anonymous referees for valuable suggestions, which
stimulate us to weaken  assumptions in section~\ref{s_infsup} and lead to a better presentation.

% ----------------------------------------
\bibliography{Bibliography_BGSS}{}
\bibliographystyle{abbrv}
% ----------------------------------------

\appendix

\section{Proof of Lemma  \ref{lemmatrace}}

First we state a result found for example in \cite{grisvard2011elliptic,brenner2007mathematical} that makes use that $\Gamma$ is Lipschitz.

\begin{proposition}\label{traceprop}
There exists a constant $C$
\begin{equation*}
\|v\|_{L^2(\Gamma)}^2 \le C \|v\|_{H^1(\Omega)} \|v\|_{L^2(\Omega)}  \quad \text{ for all } v \in H^1(\Omega).
\end{equation*}
\end{proposition}

For the moment we assume the following result.
\begin{lemma}\label{extensionT}
Let $T \in \mathcal{T}_h$. There exists an extension operator $R_T: H^1(T) \rightarrow H^1(R^d)$ such that
\begin{alignat}{1}
R_T v=& v \quad \text{ on } T \label{extensionT1}\\
\|R_T v\|_{L^2(R^d)} + h_T \|\nabla R_T v\|_{L^2(R^d)} \le &   C (\|v\|_{L^2(T)} + h_T \|\nabla v\|_{L^2(T)}) \label{extensionT2}
\end{alignat}
where the constant $C$ is independent of $T$ and $v$.
\end{lemma}

\subsection{ Proof of Lemma \ref{lemmatrace}}
Let $T \in \mathcal{T}_h$ and let $v \in H^1(T)$. Then, we have using Proposition \ref{traceprop}
\begin{equation*}
\begin{split}
\|v\|_{L^2(T \cap\Gamma)} &\le \|R_T v\|_{L^2(\Gamma)} \le C \|R_T v\|_{L^2(R^d)}^{1/2}  \| R_T v\|_{H^1(R^d)}^{1/2} \\
&\le  C \left(\|R_T v\|_{L^2(R^d)}^{1/2}   \|\nabla R_T v\|_{L^2(R^d)}^{1/2}+ \|R_T v\|_{L^2(R^d)}\right).
\end{split}
\end{equation*}
We apply the arithmetic-geometric mean inequality {and use $h_T\le h_0$} to get
\begin{equation*}
\|v\|_{L^2(T)} \le C ( h_T^{-1/2} \|R_T v\|_{L^2(R^d)}  + h_T^{1/2} \|\nabla R_T v\|_{L^2(R^d)}).
\end{equation*}
The result now follows after applying Lemma \ref{extensionT}.

\subsection{Proof of Lemma \ref{extensionT}}
We will denote the reference tetrahedra of unit size with a vertex at the origin $\hat{T}$. Then, we know (\cite{SteinBook}) there exists an extension operator from $R: H^1(\hat{T}) \rightarrow H_0^1(B_{2})$ such that
\begin{alignat}{1}
R \hat{v}= & \hat{v} \quad  \quad \text{ on } \hat{T} \label{refext1}\\
\|R \hat{v}\|_{H^1(B_2)} \le & C \|\hat{v}\|_{H^1(\hat{T})}.  \label{refext2}
\end{alignat}
Here $B_2$ is the ball with radius 2 centered at the origin.

Let $F_T: \hat{T} \rightarrow T$ be the onto affine mapping and has the form $F_T (\hat{x})= B \hat{x}+ b$. For any $v \in H^1(T)$ we can define $\hat{v} \in H^1(\hat{T})$ in the following way:
$\hat{v}(\hat{x})=v(F_T(\hat{x}))$.

Our desired extension will be given by
\begin{equation*}
(R_T v)(x)= (R \hat{v})(F^{-1}_T(x))
\end{equation*}
For notational convenience we use $w=R_T v$. Then, we see that $\hat{w}=R \hat{v}$. Using a change of variables formula we get
\begin{equation*}
\|\nabla w\|_{L^2(R^d)}^2= \int_{ F(B_2)}  |\nabla w(x)|^2  dx= \int_{B_2} |B^{-t} \nabla \hat{w}(\hat{x})|^2 |\text{det} B|  d\hat{x}.
\end{equation*}
Using that the mesh is shape regular we have (see \cite{ciarlet2002finite})
$
|B_{ij}| \le  C\, h_T $ ,$
|B_{ij}^{-1} | \le  C \, h_T^{-1}.
$
Therefore, we obtain
\begin{equation*}
\int_{B_2} |B^{-t} \nabla \hat{w}(\hat{x})|^2 |\text{det}| B  d\hat{x} \le C h_T^{d-2} \|\nabla \hat{w} \|_{L^2(B_2)}^2= C h_T^{d-2} \|\nabla R \hat{v} \|_{L^2(B_2)}^2.
\end{equation*}
Using \eqref{refext2} we obtain
\begin{equation*}
\|\nabla w\|_{L^2(R^d)}^2 \le C \, h_T^{d-2} (\|\hat{v}\|_{L^2(\hat{T})}^2+ \|\nabla \hat{v}\|_{L^2(\hat{T})}^2).
\end{equation*}
It is standard to show, again using a change of variable formula, and the bounds for $B$ and $B^{-1}$ above that
\begin{equation*}
 h_T^{d-2} (\|\hat{v}\|_{L^2(\hat{T})}^2+ \|\nabla \hat{v}\|_{L^2(\hat{T})}^2) \le C \, (h_T^{-2} \|v\|_{L^2(T)}^2+ \|\nabla v\|_{L^2(T)}^2).
\end{equation*}
Therefore, we have shown
\begin{equation*}
h_T \|\nabla R_T v\|_{L^2(R^d)} \le (\|v\|_{L^2(T)}+ h_T\|\nabla v\|_{L^2(T)}).
\end{equation*}
The bound for $\|R_T v\|_{L^2(R^d)}$ follows a similar argument.

\section{Proof of Lemma \ref{lemmaextension}}

Let $q \in Q_h$. Note that the extension $E_hq$ does not have to be continuous even if $Q_h= Q_h^{\rm cont}$.  Now for every $T \in \mathcal{T}_h$ we let $q_T^{\text{ext}} \in P^{k_p}(\mathbb{R}^d)$ be the natural  extension of $q_T \equiv q|_T$ onto the entire $\mathbb{R}^d$.

For $T \in \mathcal{T}_h^\Gamma$, we define $E_h q|_T= q_{K_T}^\text{ext}|_{T}$, where $K_T\in \Omega_h^i$ is given by assumption~\ref{assumption1} {(see the remark right below the assumption)}. %Note that there might be multiple $T$ but we can choose an arbitrary one.
Since $K_T\in W(T)~\Rightarrow~dist(K_T,T)\le C h_T$, it follows that
\[\|\nabla E_h q\|_{L^2(T)}=\|\nabla q_{K_T}^\text{ext}\|_{L^2(T)} \le C \|\nabla q\|_{L^2(K_T)}.\]
Hence, we have
\begin{equation}\label{inq201}
\sum_{T \in \mathcal{T}_h^\Gamma} h_T^2  \|\nabla E_h q\|_{L^2(T)}^2  \le C \, \sum_{T \in \mathcal{T}_h^i} h_T^2 \|\nabla q\|_{L^2(T)}^2.
\end{equation}

To bound the face terms, we let $F \in \mathcal{F}_h^\Gamma$ where $F = \partial T \cap \partial \widetilde{T}$. If we use the notation $K=K_T$ and $\widetilde{K}=K_{\widetilde{T}}$ belonging to $\mathcal{T}_h^i$ we have $E_h q|_T=q_K^{\text{ext}}|_T$ and $E_h q|_{\widetilde{T}}=q_{\widetilde{K}}^{\text{ext}}|_{\widetilde{T}}$.  Now 
 due to the assumption~\ref{assumption1b} there exists a sequence of tetrahedra $K=K_1, K_2, \ldots, K_M=\widetilde{K}$ all belonging to $\mathcal{T}_h^{i}$ where $K_i, K_{i+1} $ share a common face which we denote by $F_i$ {and}  the number $M$ is bounded and only depends on the shape regularity of the mesh.

First using inverse estimates we get
\begin{equation*}
h_F^{1/2} \|[E_h q]\|_{L^2(F)}=h_F^{1/2}\| q_{K_1}^\text{ext}-q_{K_M}^\text{ext}\|_{L^2(F)} \le  C \| q_{K_1}^\text{ext}-q_{K_M}^\text{ext}\|_{L^2(T)}
\end{equation*}
Easy to see that  since $K_1$ and $T$ belong to the same patch $W(T)$ that
\begin{equation*}
\| q_{K_1}^\text{ext}-q_{K_M}^\text{ext}\|_{L^2(T)} \le C \,   \| q_{K_1}^\text{ext}-q_{K_M}^\text{ext}\|_{L^2(K_1)}
\end{equation*}
Thanks to the triangle inequality we get
\begin{equation*}
\| q_{K_1}^\text{ext}-q_{K_M}^\text{ext}\|_{L^2(K_1)} \le \| q_{K_1}-q_{K_2}^\text{ext}\|_{L^2(K_1)}+\| q_{K_2}^\text{ext}-q_{K_M}^\text{ext}\|_{L^2(K_1)}.
\end{equation*}
Using equivalence of norms in finite dimensional case we obtain
\begin{equation*}
 \| q_{K_1}-q_{K_2}^\text{ext}\|_{L^2(K_1)}\le {C} \left(h_{F_1}^{1/2} \| [q]\|_{L^2(F_1)}+  h_{K_1} \| \nabla(q_{K_1}-q_{K_2}^\text{ext})\|_{L^2(K_1)}\right).
\end{equation*}
We also have
\begin{equation*}
\| q_{K_2}^\text{ext}-q_{K_M}^\text{ext}\|_{L^2(K_1)} \le  C\, \| q_{K_2}^\text{ext}-q_{K_M}^\text{ext}\|_{L^2(K_2)}.
\end{equation*}
So we get,
\begin{equation*}
\| q_{K_1}^\text{ext}-q_{K_M}^\text{ext}\|_{L^2(K_1)} \le {C} \left(h_{F_1}^{1/2} \| [q]\|_{L^2(F_1)}+  h_{K_1} \| \nabla(q_{K_1}-q_{K_2}^\text{ext})\|_{L^2(K_1)}\right)+\| q_{K_2}^\text{ext}-q_{K_M}^\text{ext}\|_{L^2(K_2)}.
\end{equation*}
If we continue this we will get
\begin{equation*}
\| q_{K_1}^\text{ext}-q_{K_M}^\text{ext}\|_{L^2(K_1)} \le  {C} \left(\sum_{j=1}^{M-1} h_{F_j}^{1/2} \| [q]\|_{L^2(F_j)}+\sum_{j=1}^{M-1} h_{K_j} \| \nabla(q_{K_j}-q_{K_{j+1}}^\text{ext})\|_{L^2(K_j)}\right).
\end{equation*}
Again, we see that
\begin{equation*}
\sum_{j=1}^{M-1} h_{K_j} \| \nabla(q_{K_j}-q_{K_{j+1}}^\text{ext})\|_{L^2(K_j)} \le {C}\sum_{j=1}^{M} h_{K_j}  \|\nabla q\|_{L^2(K_j)}.
\end{equation*}
Hence, we get
\begin{equation*}
h_F^{1/2} \|[E_h q]\|_{L^2(F)} \le C  \left(\sum_{j=1}^{M-1} h_{F_j}^{1/2} \| [q]\|_{L^2(F_j)}+ \sum_{j=1}^{M} h_{K_j}  \|\nabla q\|_{L^2(K_j)}\right).
\end{equation*}

%A similar argument works if $T \in \mathcal{T}_h^\Gamma$ and $\widetilde{T} \in \mathcal{T}_h^i$. In that case we let $\widetilde{K}=\widetilde{T}$ and same inequality above holds true.

If we now sum over $F \in \mathcal{F}_h^\Gamma$ we get
\begin{equation*}
\sum_{F \in \mathcal{F}_h^\Gamma}  h_F \|[E_h q]\|_{L^2(F)}^2 \le C \, \sum_{T \in \mathcal{T}_h^{i}} h_T^2 \|\nabla q\|_{L^2(T)}^2+ C \, \sum_{F \in \mathcal{F}_h^{i}}  h_F \|[ q]\|_{L^2(F)}^2.
\end{equation*}
The result now follows by combining this inequality with \eqref{inq201}.

\end{document}